\documentclass[a4paper,10pt]{article}
\usepackage[utf8x]{inputenc}
\usepackage{graphicx}
\usepackage{color}
\usepackage[colorlinks]{hyperref}
\usepackage{lscape}
\usepackage[centertags]{amsmath}
\usepackage{amssymb}
\usepackage{amsthm}
\usepackage{newlfont}
\usepackage{amsfonts}
\theoremstyle{plain}
\newtheorem{thm}{Theorem}
\newtheorem{cor}{Corollary}

\newtheorem{prop}{Proposition}
\newtheorem{conjec}{conjecture}
\newtheorem{defn}{Definition}

\begin{document}

\begin{center}\Large \textbf{Seymour's second neighborhood conjecture for tournaments missing a generalized star}
\end{center}
\begin{center}
Salman GHAZAL\footnote{Department of Mathematics, Faculty of Sciences I, Lebanese University, Hadath, Beirut, Lebanon.\\
                       E-mail: salmanghazal@hotmail.com\\
                       Institute Camille Jordan, Département de Mathématiques, Universit\'{e} Claude Bernard Lyon 1, France.\\
                       E-mail: salman.ghazal@math.univ-lyon1.fr
                       
                       }
\end{center}
\vskip1cm

\begin{abstract}
Seymour's Second Neighborhood Conjecture asserts that every digraph (without digons) has a vertex whose first out-neighborhood
is at most as large as its second out-neighborhood. We prove its weighted version for tournaments missing a generalized star.
As a consequence the weighted version holds for tournaments missing a sun, star, or a complete graph.
\end{abstract}

\begin{section}{Introduction}

\par \hskip.6cm In this paper, graphs are finite and simple. Directed graphs (digraphs) are orientations of graphs, so they do not contain loops,
parallel arcs, or digons (directed cycles of length 2). The neighborhood of a vertex $v$ in a graph $G$ is denoted by $N_G(v)$ and its degree is
$d_G(v)=|N_G(v)|$. Let $D=(V,E)$ denote a digraph with vertex set $V$ and arc set $E$. As usual, $N^{+}_D(v)$ (resp. $N^{-}_D(v)$) denotes the (first)
out-neighborhood (resp. in-neighborhood) of a vertex $v\in V$. $N^{++}_D(v)$ (resp. $N^{--}_D(v)$) denotes the second out-neighborhood 
(in-neighborhood) of $v$, which is the set of vertices that are at distance 2 from $v$ (resp. to $v$). We also denote 
$d^{+}_D(v)=|N^{+}_D(v)|$, $d^{++}_D(v)=|N^{++}_D(v)|$, $d^{-}_D(v)=|N^{-}_D(v)|$ and $d^{--}_D(v)=|N^{--}_D(v)|$. We omit the subscript if the
digraph (resp. graph) is clear from the context. For short, we write $x\rightarrow y$ if the arc $(x,y)\in E$. 
We say that a vertex $v$ has the \emph{second neigborhood property} (SNP) if $d^{+}(v)\leq d^{++}(v)$.\\

\par In 1990, P. Seymour conjectured \cite{dean} the following statement: 

\begin{conjec}
 \textbf{(The Second Neighborhood Conjecture (SNC))}\\ Every digraph has a vertex with the SNP.
\end{conjec}

Seymour's conjecture restricted to tournaments is known as Dean's conjecture \cite{dean}. In 1996, Fisher \cite{fisher}
proved Dean's conjecture, thus asserting the SNC for tournaments. Another proof of Dean's conjecture was given by Thomassé
and Havet \cite{m.o.}, in 2000, using a tool called (local) median order. This tool turned out to be useful not only for the proof of SNC for 
tournaments but also for Sumner's conjecture \cite{sumn} ( see \cite{m.o.}, \cite{sahi}). Yuster and Fidler \cite{fidler}, in 2007, also used
(weighted) local median order to prove the SNC for tournaments missing the edges of a complete graph and tournaments missing a matching.\\

\par Chen, Shen and Yuster \cite{shen} proved that every digraph contains a vertex $v$ such that $d^{+}(v)\leq \gamma d^{++}(v)$, 
 where $\gamma = 0.657298...$ is the unique real root of the equation $2x^3+x^2-1=0$. In addition, Kaneko and Locke \cite{kaneko} proved the
SNC for digraphs with minimum out-degree at most 6.\\

\par For completeness, we introduce the following related conjecture, which was proposed in 1978 by Caccetta and H\"{a}ggkvist \cite{cachag}.

\begin{conjec}
 If D is a digraph with minimum out-degree at least $|V(D)|$/k, then D has a directed cycle of length at most k. 
\end{conjec}

SNC, if true, will establish an important special case ($k=3$) of conjecture $2$, which is still open.\\

\par Let $D=(V,E)$ be a digraph (vertex) weighted by a non-negative real valued function \hskip.1cm$\omega :V\rightarrow \mathcal{R_+}$. The weight
of an arc $(x,y)$ is the weight of its head $y$. The weight of a set of vertices (resp. edges) is the sum of the weights of its members.
We say that a vertex $v$ has the weighted SNP if $\omega(N^+(v))\leq \omega(N^{++}(v))$. It is known that the SNC is equivalent to its weighted
version:\emph{ Every weighted digraph has a vertex with the weighted SNP}.\\

\par A weighted median order $L=v_1v_2...v_n$ of a weighted digraph $(D,\omega)$ is an order of the vertices of $D$ the maximizes
the weight of the set of forward arcs of $D$, i.e., the set $\{(v_i,v_j)\in E; i<j\}$.
In fact, $L$ satisfies the feedback property: For all $1\leq i\leq j\leq n:$
$$ \omega ( N^{+}_{[i,j]}(v_i) ) \geq  \omega ( N^{-}_{[i,j]}(v_i) ) $$
and 
$$ \omega ( N^{-}_{[i,j]}(v_j) ) \geq  \omega ( N^{+}_{[i,j]}(v_j) ) $$
where $[i,j]:=D[v_i,v_{i+1}, ...,v_j]$.\\

\par An order $L=v_1v_2...v_n$ satisfying the feedback property is called weighted local median order. When $\omega = 1$, we obtain
     the defintion of (local) median orders of a digraph (\cite{m.o.}, \cite{fidler}). The last vertex $v_n$
     of a weighted local median order $L=v_1v_2...v_n$ of $(D,\omega)$ is called a \emph{feed} vertex of the weighted digraph 
     $(D,\omega)$.

\begin{thm}\cite{m.o.}
 Every feed vertex of a tournament has the SNP.
\end{thm}
Following the proof in \cite{m.o.}, the weighted version of the previous statement is proved in \cite{fidler}.
\begin{prop}
Every feed vertex of a weighted tournament has the weighted SNP.
\end{prop}

\par The above weighted version was used in \cite{fidler} to prove the SNC for tournaments missing a matching.
In the next section, we will introduce the definition of n-generalized star and characterize it, to prove the weighted
version of SNC for tournaments missing generalized star, thus generlizing a result of \cite{fidler}. In fact, we prove a more general statement
(Theorem 2). As corollaries, the weighted SNC holds for tournaments missing a sun, star or a complete graph.

\end{section}

\begin{section}{Main Results}

Let $D=(V,E)$ be a digraph. For 2 vertices $x$ and $y$, we call $xy$ a missing edge if $(x,y) \notin E$ and
$(y,x)\notin E$. The missing graph $G$ of $D$ is the graph formed by the missing edges, formally, $E(G)$ is the set
of all the missing edge and $V(G)$ is the set of non whole vertices (vertices incident to some missing edges).

\begin{defn}
A missing edge $ab$ is called \emph{good} if:\\
$(i)$   $(\forall v \in V\backslash\{a,b\})[(v\rightarrow a)\Rightarrow(b\in N^{+}(v)\cup N^{++}(v))]$ or\\
$(ii)$ $(\forall v \in V\backslash\{a,b\})[(v\rightarrow b)\Rightarrow(a\in N^{+}(v)\cup N^{++}(v))]$.\\
If $ab$ satisfies $(i)$ we say that $(a,b)$ is a convenient orientation of $ab$.\\
If $ab$ satisfies $(ii)$ we say that $(b,a)$ is a convenient orientation of $ab$.
\end{defn}

The definition of good missing edges is inspired from \cite{fidler} (subsection 3.1).

\begin{thm}

Let $(D,\omega)$ be a weighted digraph. If all the missing edges of $D$ are good then it has a vertex with the weighted SNP.

\end{thm}

\begin{proof}
We give every missing edge a convenient orientation and add it to $D$. The obtained digraph is a tournament $T$.
Consider a weighted local median order $L$ of $(T,\omega)$ and let $f$ denote its feed vertex. We modify $T$ by reorienting all the missing edges
incident to $f$ towards $f$, if any exists. Let $T'$ denote the new obtained tournament. $L$ is also a weighted local median order of $(T',\omega)$.
We have that $f$ has the weighted SNP in $T'$, by proposition 1. Note that $N^+(f)=N^+_{T'}(f)$. Suppose $f\rightarrow u\rightarrow v$ in $T'$.
Either $(u,v)\in E(D$) or a convenient orientation. Thus $v\in N^{+}(f)\cup N^{++}(f)$. Whence, $N^{++}(f)=N^{++}_{T'}(f)$. Therefore $f$ has the
weighted SNP in $(D,\omega)$ as well.
\end{proof}

\begin{defn}
An n-generalized star $G_n$ is a graph defined as follows:\\
\begin{description}
\item[1) ] $V(G_n)=\displaystyle\bigcup_{i=1}^{n}(X_i\cup A_{i-1})$, where the $A_i$'s and $X_i$'s are pairwise disjoint sets
\item[2) ] $G_n[\displaystyle\bigcup_{i=1}^{n}X_i]$ is a complete graph and $X_i$'s are nonempty
\item[3) ] $\displaystyle\bigcup_{i=1}^{n}A_{i-1}$ is a stable set and $A_i$ is nonempty for all $i>0$
\item[4) ] $N(A_0)=\phi$ and for all $i>0$, for all $a\in A_i$, $N(a)=\displaystyle\bigcup_{1\leq j\leq i}X_j$.
\end{description}
\end{defn}

A \emph{sun} $G$ is a graph formed of a complete graph $T$ and a stable set $S$ such that for every $s\in S$  we have $N(s)=V(T)$.
Clearly, $G$ is a 2-generalized star or a 1-generalized star. If $V(T)$ is a singleton then $G$ is a star and if $S$ is empty then $G$ is
a complete graph.\\

Recall that a sqaure is a cycle of length 4. The following theorem shows 2 characterizations of generalized stars.
The first one is structural, while the second characterizes them when they are considerd as missing graphs of digraphs.
\newpage
\begin{thm}
Let G be a simple graph. The following are equivalent:
\begin{description}
  \item[(A)] G is a generalized star.
  \item[(B)] Any two nonadjacent edges of G do not induce a subgraph of square.
  \item[(C)] All the missing edges of every digraph whose missing graph is G are good missing edges.
\end{description}
\end{thm}

\begin{proof}

$A\Rightarrow B$: By the definition of a generalized star.\\

$B\Rightarrow A$: By setting $A_0$ the set of isolated vertices, we may assume that $G$ has no isolated vertices. Let $S$ be a stable set
in $G$ with the maximum size. Set $T=V(G)-S$. We show that $T$ is a clique. By the maximality of $S$, every element of $T$ has a neighbor
in $S$. Suppose $x,y \in T$. If $(N(x)\bigcup N(y))\cap S=\{a\}$ then $xy\in E(G)$, since otherwise the stable set $S\bigcup\{x,y\}-\{a\}$ is
lager than $S$ which is a contradiction. Otherwise, there's distinct vertices $a,\,b\in S$ such that $ax$ and $by$ are in $E(G)$. By hypothesis
these two edges do not induce a subgraph of a square, then at least one of them has at least an endpoint which is adjacent to the endpoints of
the other. Assume, without loss of generality this edge is $ax$. Since $S$ is stable, $x$
is the endpoint which is adjacent to $b$ and $y$. In particular, $xy\in E(G)$. Thus $T$ is a clique.\\
Suppose $a,b\in S$ with $d(a)\leq d(b)$. We prove $N(a)\subseteq N(b)$. Suppose there is $x\in N(a)-N(b)$. Since
$d(a)\leq d(b)$ there is $y\in N(b)-N(a)$. Thus the path $axyb$ is the induced graph in $G$ by the two nonadjacent edges $ax$ and $by$, 
which is a subgraph of a square, a contradiction. Whence, $N(a)\subseteq N(b)$.
Finally, let $d_1<...<d_s$ be the list of distinct degrees of vertices of $S$. Set $A_i=\{a\in S; d(v)=d_i\}$ and
$X_i=\{x\in T;$ there is $a \in A_i$ such that $ax\in E(G)\}\backslash \displaystyle\bigcup _{j<i}X_j$. From these two families of sets we
can show that $G$ is an $s$ or $s+1$-generalized star.\\

$B\Rightarrow C$: Let $D$ be a digraph whose missing graph is $G$ and let $ab$ be a missing edge. Suppose, to the contrary, that $ab$ is not
good. Then there is $u,v\in V(D)-\{a,b\}$ such that $u\rightarrow a$, $b\notin N^{+}(u)\cup N^{++}(u)$, $v\rightarrow b$ and 
$a\notin N^{+}(v)\cup N^{++}(v)$. In this case, also $uv$ is a missing edge and not adjacent to $ab$. Clearly, These 2 missing
edges induce a subraph of a square. A contradiction. \\

$C\Rightarrow B$: Suppose to the contrary that there are two nonadjacent edges in $G$, say $xy$ and $uv$, that induce in $G$ a subgraph of square. 
We may assume without lose of generality that $xu$ and $yv$ are not in $E(G)$. We construct a digraph $D$ whose missing graph is $G$ and such that
$xy$ is not good as follows: $V(D)=V(G)$. For a vertex $w$ with $wu\notin E(G)$ (resp. $wv\notin E(G)$), $(w,u)\in E(D)$ (resp. $(w,v)\in E(D)$),
with exception when $w=x$ (resp. $w=y$), $(u,x)\in E(D)$ (resp. $(v,y)\in E(D)$). For any two nonadjacent vertices $w,t$ in $G$ both not in 
$\{u,v\}$, we give $wt$ any orientation to be in $E(D)$. By construction of $D$, $u\rightarrow x$, $y\notin N^{+}(u)\cup N^{++}(u)$, 
$v\rightarrow y$ and $x\notin N^{+}(v)\cup N^{++}(v)$. Whence, $xy$ is not a good missing edge of $D$. A contradiction.

\end{proof}

Now, theorems 2 and 3 imply the following statements.

\begin{cor}
 Every weighted digraph whose missing graph is a generalized star has a vertex with the weighted SNP.
\end{cor}

\begin{cor}
 Every weighted digraph whose missing graph is a sun has a vertex with the weighted SNP.
\end{cor}

\begin{cor}
 Every weighted digraph whose missing graph is a star has a vertex with the weighted SNP.
\end{cor}

The proof of the non weighted version of the above fact appearing as Theorem 3.5 in \cite{fidler}
has a minor error.

\begin{cor}
 Every weighted digraph whose missing graph is a complete graph has a vertex with the weighted SNP.
\end{cor}

The non weighted version of the above corollary was already proved in \cite{fidler}.

\end{section}

\textbf{Acknowledgement.}
I thank Pr. A. El Sahili for a useful discussion.

\end{document}